\documentclass[12pt,a4paper]{article}
\usepackage[latin1]{inputenc}
\usepackage{amsmath, amsthm}
\usepackage{amsfonts}
\usepackage{amssymb}
\usepackage{makeidx}
\usepackage{graphicx}
\newtheorem{definition}{Definition}[section]
\newtheorem{proposition}{Proposition}[section]
\newtheorem{theorem}{Theorem}[section]

\newtheorem{lemma}{Lemma}[section]

\usepackage[width=150.00mm, height=195.00mm]{geometry}
\usepackage{color}
\usepackage{hyperref}
\author{Ratan Lal, Alka Choudhary and Vipul Kakkar}
\title{Determinant for automorphisms of semidirect product of groups}

\begin{document}

	\maketitle
	\begin{flushleft}
		\textrm{Desh Bhagat Pandit Chetan Dev Government College of Education, Faridkot}\\
	\textrm{Manipal University, Jaipur, Rajasthan, India}\\
	\textrm{Central University of Rajasthan, Ajmer, Rajasthan, India}\\
\end{flushleft}
	E-mails: \url{vermarattan789@gmail.com}, \url{alkababal@gmail.com}, \url{vplkakkar@gmail.com}
	
	\abstract{A description of the endomorphisms of semidirect products of two groups as a group of $2\times 2$ matrices of maps is already known. Using this description, we have studied the concept of determinant for the endomorphisms of semidirect product of two groups. A characterization of the invertible endomorphisms is given with the help of the tools developed using the determinants.}\\
		
\noindent	\rm{Keywords: Semidirect product, Determinant, Automorphism group}
	\section{Introduction}
	
	The study of automorphism groups of groups have been an area of interest for many algebraists. The automorphism group of direct products \cite{bcm, jb}, semidirect products \cite{cur, jd, hsu} and Zappa-Sz\'{e}p products \cite{zapfix, zap} have been studied as the group of $2\times 2$ matrices of maps that satisfy some specific conditions. 
	
	In \cite{mb}, the authors have studied matrix characterizations of the automorphism groups and given some algorithms for an endomorphism of direct product of groups to be invertible. Let $\begin{pmatrix}
	\alpha & \beta \\ \gamma & \delta
	\end{pmatrix}$ be a matrix associated to an endomorphism $\phi$ of direct product of two groups $H$ and $K$ (see \cite{mb} for details). If $\alpha$ is invertible, then the authors have associated a map called the determinant of $\phi$ and defined as $\det_{H}(\phi) = \alpha - \beta\delta^{-1}\gamma$. Similarly, if $\delta$ is invertible, then they have defined the determinant of $\phi$  as $\det_{K}(\phi) = \delta - \gamma\alpha^{-1}\beta$.

	Note that, in the usual matrices over a field, a characterization for a matrix to be invertible is that its determinant should be non-zero or in other words, the determinant is invertible. In \cite{mb}, the authors have extended this characterization and proved that an endomorphism of direct product of groups is invertible if and only if the determinant of the associated matrix is invertible.
	
	Let $H$ and $K$ be two groups and $f: K \longrightarrow Aut(H)$ a group homomorphism defined by $f(k)(h) = f_{k}(h)$ for all $h\in H$ and $k\in K$, where $Aut(H)$ is the automorphism group of $H$. Then the set $H\times K = \{(h, k)\mid h\in H, k\in K\}$ forms a group with the binary operation defined by 
	\begin{equation*}
	(h_{1}, k_{1})(h_{2}, k_{2}) = (h_{1}f_{k_{1}}(h_{2}), k_{1}k_{2}).
	\end{equation*}
	This group is called the external semidirect product of groups and is denoted by $H\rtimes_{\phi} K$. When it is clear from the description of the semidirect product, we will denote it by $H\rtimes K$. On the other hand, let $G$ be a group. Let $H$ be a normal subgroup of $G$ and $K$ be a subgroup of $G$ such that $H \cap K = \{1\}$ and $G = HK$. The left action of $H$ on $K$ is defined as $h^{k} = khk^{-1}$ for all $h\in H$ and $k\in K$. Then $G$ is called the internal semidirect product of $H$ and $K$. We will identify the internal semidirect product with the external semidirect product.

	Note that the semidirect product of two groups is the natural generalization of the direct product of two groups. This motivates us to study the determinants for the semidirect product of two groups. In this paper, we have defined both the determinants $\det_{H}$ and $\det_{K}$ for an endomorphism of semidirect product of two groups $H$ and $K$. In section \ref{s2}, we have given an isomorphism between the endomorphism group of semidirect product of two groups and the monoid consisiting of $2\times 2$ matrices of maps satisfying some conditions (see Theorem \ref{t1}). In section \ref{s3}, we have defined both the determinants $\det_{H}$ and $\det_{K}$ for a matrix associated with an endomorphism. Also, we have a characterization for an endomorphism of semidirect product of two groups to be invertible in the Theorem \ref{t32}. Further, we have proved that for an automorphism of semidirect product of two groups $H$ and $K$, the determinant $\det_{H}$ is invertible if and only if  $\det_{K}$ is invertible.
	
	Let $H$ and $K$ be two groups. Then, throughout the paper $Hom(H, K)$, $End(H)$ and $CHom(H, K)$ denote the group of all the homomorphisms from $H$ to $K$,  the group of all the endomorphisms of $H$ and the group of all the crossed homomorphisms from $H$ to $K$ respectively. $Z(H)$ will denote the center of the group $H$.  
	
	\section{Preliminaries}\label{s2}
	Let $U$, $V$ and $W$ be any groups. Then $Map(U, V)$ denotes the set of all maps between the groups $U$ and $V$. If $\phi, \psi \in Map(U, V)$ and $\eta \in Map(V, W)$, then $\phi + \psi \in	Map(U, V)$ is defined by $(\phi + \psi)(u) = \phi(u)\psi(u)$, $\eta \phi\in Map(U, W)$ is defined by $\eta\phi(u) = \eta(\phi(u))$, $-\phi \in Map(U, V)$ is defined as $-\phi(u) = \phi(u)^{-1}$ and $\phi^{\psi}\in Map(U, V)$ is defined by $(\phi^{\psi})(u) = \phi(u)^{\psi(u)} = \psi(u)\phi(u)\psi(u)^{-1}$ for all $u\in U$.

	Let $H$ and $K$ be groups. Then consider the set
	\begin{equation*}
	\mathcal{M} = \left\{\begin{pmatrix}
	\alpha & \beta\\ \gamma & \delta
	\end{pmatrix}
	\mid \begin{matrix}
	\alpha \in Map(H, H), & \beta\in Map(K, H)\\ \gamma\in Hom(H, K), & \text{and}\;\delta \in End(K)
	\end{matrix} \right\},
	\end{equation*}
	where the maps $\alpha$, $\beta$, $\gamma$ and $\delta$ satisfy the following
	\begin{itemize}
		\item[$(i)$] $\alpha(hh^{\prime}) = \alpha(h)\alpha(h^{\prime})^{\gamma(h)}$,
		\item[$(ii)$] $\beta(kk^{\prime}) = \beta(k)\beta(k^{\prime})^{\delta(k)}$,
		\item[$(iii)$] $\gamma(h^{k})\delta(k) = \delta(k)\gamma(h)$,
		\item[$(iv)$] $\alpha(h^{k})\beta(k)^{\gamma(h^{k})} = \beta(k)\alpha(h)^{\delta(k)}$ 
	\end{itemize}
	for all $h,h^{\prime}\in H$ and $k, k^{\prime}\in K$. Then the set $\mathcal{M}$ forms a monoid with the binary operation defined as
	\[\begin{pmatrix}
	\alpha^{\prime} & \beta^{\prime}\\ \gamma^{\prime} & \delta^{\prime}
	\end{pmatrix}\begin{pmatrix}
	\alpha & \beta\\ \gamma & \delta
	\end{pmatrix} = \begin{pmatrix}
	\alpha^{\prime}\alpha + (\beta^{\prime}\gamma)^{\gamma^{\prime}\alpha} & \alpha^{\prime}\beta + (\beta^{\prime}\delta)^{\gamma^{\prime}\beta}\\ \gamma^{\prime}\alpha + \delta^{\prime}\gamma & \gamma^{\prime}\beta + \delta^{\prime}\delta
	\end{pmatrix}  \]
	and the identity element as $\begin{pmatrix}
	1 & 0\\ 0 & 1
	\end{pmatrix}$, where $1$ denotes the identity morphism and $0$ denotes the trivial morphism.

	\begin{theorem}\label{t1}
		Let $H$ and $K$ be groups. Then $End(H\rtimes K) \simeq \mathcal{M}$ as monoids. 
	\end{theorem}
	\begin{proof}
		Let $G = H\rtimes K$ and $\theta\in End(G)$. Then for all $h\in H$ and $k\in K$, we define the maps $\alpha$, $\beta$, $\gamma$ and $\delta$ as $\theta(h) = \alpha(h)\gamma(h)$ and $\theta(k) = \beta(k)\delta(k)$. Now, for all $h, h^{\prime}\in H$, we have
		\begin{align*}
		\alpha(hh^{\prime})\gamma(hh^{\prime}) &= \theta(hh^{\prime})\\
		 &= \theta(h)\theta(h^{\prime})\\
		  &= \alpha(h)\gamma(h)\alpha(h^{\prime})\gamma(h^{\prime})\\
	&= \alpha(h)\alpha(h^{\prime})^{\gamma(h)}\gamma(h)\gamma(h^{\prime}).
		\end{align*}
		  Using the uniqueness of representation, we get $\alpha(hh^{\prime}) = \alpha(h)\alpha(h^{\prime})^{\gamma(h)}$ and $\gamma(hh^{\prime}) = \gamma(h)\gamma(h^{\prime})$. Now, for all $k,k^{\prime}\in K$, we have
		  \begin{align*}
		  \beta(kk^{\prime})\delta(kk^{\prime}) &= \theta(kk^{\prime})\\
		  &= \theta(k)\theta(k^{\prime})\\
		  &= \beta(k)\delta(k)\beta(k^{\prime})\delta(k^{\prime})\\
	&= \beta(k)\beta(k^{\prime})^{\delta(k)}\delta(k)\delta(k^{\prime}).
		  \end{align*}
		   Using the uniqueness of representation, we get $\beta(kk^{\prime}) = \beta(k)\beta(k^{\prime})^{\delta(k)}$ and $\delta(kk^{\prime}) = \delta(k)\delta(k^{\prime})$.
		
		Now, for all $h \in H$ and $k \in K$, 
		\begin{align*}
		\alpha(h^{k}) \beta(k)^{\gamma(h^{k})} \gamma(h^{k}) \delta(k) &= \alpha(h^{k}) \gamma(h^{k}) \beta(k) \delta(k)\\
		&= \theta(h^{k}k)\\
		&= \theta(kh)\\
		&= \theta(k) \theta(h)\\
		&= \beta(k) \delta(k) \alpha(h) \gamma(h)\\
		&= \beta(k) \alpha(h)^{\delta(k)} \delta(k) \gamma(h).
		\end{align*}
	 This implies that $\alpha(h^{k}) \beta(k)^{\gamma(h^{k})} = \beta(k) \alpha(h)^{\delta(k)}$ and $\gamma(h^{k}) \delta(k) = \delta(k) \gamma(h)$. Therefore, the maps $\alpha$, $\beta$, $\gamma$ and $\delta$ satisfy all the conditions $(i)-(iv)$ and so, $\begin{pmatrix}
		\alpha & \beta \\ \gamma & \delta
		\end{pmatrix} \in \mathcal{M}$.
		
		Thus to each $\theta \in End(G)$ we can associate an element $\begin{pmatrix}
		\alpha & \beta \\ \gamma & \delta
		\end{pmatrix}\in \mathcal{M}$. This defines a map $\lambda: End(G) \longrightarrow \mathcal{M}$ by
		\begin{equation*}
		\lambda(\theta) = \begin{pmatrix}
		\alpha & \beta \\ \gamma & \delta
		\end{pmatrix}.
		\end{equation*}
		Now, we will prove that the map $\lambda$ is a monoid homomorphism. Let $\theta, \theta^{\prime}\in End(G)$. Let $\alpha$, $\beta$, $\gamma$ and $\delta$ be the maps associated with $\theta$ and $\alpha^{\prime}$, $\beta^{\prime}$, $\gamma^{\prime}$ and $\delta^{\prime}$ be the maps associated with $\theta^{\prime}$. Then for all $hk\in G$, we have
		
		\begin{align*}
		&\theta^{\prime}\theta(hk)\\ 
		&= \theta^{\prime}(\theta(hk))\\
		&= \theta^{\prime}(\alpha(h)\gamma(h)\beta(k)\delta(k))\\
		&= \theta^{\prime}(\alpha(h)\beta(k)^{\gamma(h)}\gamma(h)\delta(k))\\
		&= \alpha^{\prime}(\alpha(h)\beta(k)^{\gamma(h)})\gamma^{\prime}(\alpha(h)\beta(k)^{\gamma(h)})\beta^{\prime}(\gamma(h)\delta(k))\delta^{\prime}(\gamma(h)\delta(k))\\
		&= \alpha^{\prime}(\alpha(h))\alpha^{\prime}(\beta(k)^{\gamma(h)})^{\gamma^{\prime}(\alpha(h))}\gamma^{\prime}(\alpha(h))\gamma^{\prime}(\beta(k)^{\gamma(h)})\beta^{\prime}(\gamma(h))\beta^{\prime}(\delta(k))^{\delta^{\prime}(\gamma(h))}\\
		&~~~~\delta^{\prime}(\gamma(h))\delta^{\prime}(\delta(k))\\
		&= \alpha^{\prime}\alpha(h)\gamma^{\prime}(\alpha(h))\alpha^{\prime}(\beta(k)^{\gamma(h)})\gamma^{\prime}(\beta(k)^{\gamma(h)})\beta^{\prime}(\gamma(h))\delta^{\prime}(\gamma(h))\beta^{\prime}(\delta(k))\delta^{\prime}\delta(k)\\
		&= \alpha^{\prime}\alpha(h)\gamma^{\prime}(\alpha(h))(\alpha^{\prime}(\beta(k)^{\gamma(h)})\beta^{\prime}(\gamma(h))^{\gamma^{\prime}(\beta(k)^{\gamma(h)})})(\gamma^{\prime}(\beta(k)^{\gamma(h)})\delta^{\prime}(\gamma(h)))\\
		&~~~~\beta^{\prime}(\delta(k))\delta^{\prime}\delta(k)\\
		&= \alpha^{\prime}\alpha(h)\gamma^{\prime}(\alpha(h))\beta^{\prime}(\gamma(h))\alpha^{\prime}(\beta(k))^{\delta^{\prime}(\gamma(h))}\delta^{\prime}(\gamma(h))\gamma^{\prime}(\beta(k))\beta^{\prime}(\delta(k))\delta^{\prime}\delta(k)\\
		&= \alpha^{\prime}\alpha(h)\beta^{\prime}(\gamma(h))^{\gamma^{\prime}(\alpha(h))}\gamma^{\prime}(\alpha(h))\delta^{\prime}(\gamma(h))\alpha^{\prime}(\beta(k))\beta^{\prime}(\delta(k))^{\gamma^{\prime}(\beta(k))}\gamma^{\prime}(\beta(k))\delta^{\prime}\delta(k)\\
		&= (\alpha^{\prime}\alpha + (\beta^{\prime}\gamma)^{\gamma^{\prime}\alpha})(h)(\gamma^{\prime}\alpha + \delta^{\prime}\gamma)(h)(\alpha^{\prime}\beta + (\beta^{\prime}\delta)^{\gamma^{\prime}\beta})(k)(\gamma^{\prime}\beta + \delta^{\prime}\delta)(k).
		\end{align*} 
		
		Thus, the maps associated with $\theta^{\prime}\theta\in End(G)$ are $a = \alpha^{\prime}\alpha + (\beta^{\prime}\gamma)^{\gamma^{\prime}\alpha}$, $b = \alpha^{\prime}\beta + (\beta^{\prime}\delta)^{\gamma^{\prime}\beta}$, $c = \gamma^{\prime}\alpha + \delta^{\prime}\gamma$, and $d = \gamma^{\prime}\beta + \delta^{\prime}\delta$. Also, it is easy to check that the maps $a$, $b$, $c$ and $d$ satisfy all the conditions $(A_{1}) - (A_{4})$ and so $\begin{pmatrix}
		a & b \\ c & d
		\end{pmatrix}\in \mathcal{A}$. Therefore,
		\begin{align*}
		\lambda(\theta^{\prime}\theta) =& \begin{pmatrix}
		\alpha^{\prime}\alpha + (\beta^{\prime}\gamma)^{\gamma^{\prime}\alpha}& \alpha^{\prime}\beta + (\beta^{\prime}\delta)^{\gamma^{\prime}\beta}\\  \gamma^{\prime}\alpha + \delta^{\prime}\gamma & \gamma^{\prime}\beta + \delta^{\prime}\delta
		\end{pmatrix}\\
		=& \begin{pmatrix}
		\alpha^{\prime} & \beta^{\prime} \\ \gamma^{\prime} & \delta^{\prime}	\end{pmatrix}\begin{pmatrix}
		\alpha & \beta \\ \gamma & \delta
		\end{pmatrix}\\
		=& \lambda(\theta^{\prime})\lambda(\theta).
		\end{align*} 
		Thus $\lambda$ is a homomorphism. Now, we prove that the map $\lambda$ is a bijection.
		
		Let $A = \begin{pmatrix}
		\alpha & \beta \\ \gamma & \delta
		\end{pmatrix} \in \mathcal{M}$, where the maps $\alpha$, $\beta$, $\gamma$ and $\delta$ satisfy the conditions $(i)-(iv)$. Then we define a map $\theta: G \longrightarrow G$ by $\theta(hk) = \alpha(h)\gamma(h)\beta(k)\delta(k)$ for all $h\in H$ and $k\in K$. Let $h, h^{\prime}\in H$ and $k, k^{\prime}\in K$. Then, we have
		
		\begin{align*}
		\theta((hk)(h^{\prime}k^{\prime})) &= \theta(h{h^{\prime}}^{k}kk^{\prime})\\
		&=  \alpha(h{h^{\prime}}^{k})\gamma(h{h^{\prime}}^{k})\beta(kk^{\prime})\delta(kk^{\prime})\\
		&=  \alpha(h)\alpha({h^{\prime}}^{k})^{\gamma(h)}\gamma(h)\gamma({h^{\prime}}^{k})\beta(k)\beta(k^{\prime})^{\delta(k)}\delta(k)\delta(k^{\prime})\\
		&=  \alpha(h)\gamma(h)\alpha({h^{\prime}}^{k})\gamma({h^{\prime}}^{k})\beta(k)\delta(k)\beta(k^{\prime})\delta(k^{\prime})\\
		&= \alpha(h)\gamma(h)(\alpha({h^{\prime}}^{k})\beta(k)^{\gamma({h^{\prime}}^{k})})(\gamma({h^{\prime}}^{k})\delta(k))\beta(k^{\prime})\delta(k^{\prime})\\
		&= \alpha(h)\gamma(h)\beta(k)\alpha(h^{\prime})^{\delta(k)}\delta(k)\gamma(h^{\prime})\beta(k^{\prime})\delta(k^{\prime})\\
		&= (\alpha(h)\gamma(h)\beta(k)\delta(k))(\alpha(h^{\prime})\gamma(h^{\prime})\beta(k^{\prime})\delta(k^{\prime}))\\
		&= \theta(hk)\theta(h^{\prime}k^{\prime}).
		\end{align*} 
		
		Thus $\theta \in End(G)$ such that $\lambda(\theta) = A$. This shows that $\lambda$ is a surjection.
		
		Now, let $\theta \in \ker(\lambda)$ be an element and $\alpha$, $\beta$, $\gamma$ and $\delta$ be the maps associated with $\theta$. Then $\begin{pmatrix}
		1 & 0 \\ 0 & 1
		\end{pmatrix} = \lambda(\theta) = \begin{pmatrix}
		\alpha & \beta \\ \gamma & \delta
		\end{pmatrix}$. This implies that $\alpha = 1$, $\beta = 0$, $\gamma = 0$ and $\delta = 1$. Therefore, $\theta(hk) = \alpha(h)\gamma(h)\beta(k)\delta(k) = hk$. Thus $\ker(\lambda) = \{I_{G}\}$, where $I_{G}$ is the identity map on $G$. This shows that the map $\lambda$ is a bijection. Hence, $End(G) \simeq \mathcal{M}$ as monoid.
	\end{proof}
	From here onwards, we will identify the endomorpisms of the semidirect product of groups $H$ and $K$ with the elements of $\mathcal{M}$.
	\section{Determinants}\label{s3}
	In this section, we will define the determinants associated with an endomorphism of semidirect product of groups. Let $H$ and $K$ be two groups and $\theta = \begin{pmatrix}
	\alpha & \beta\\ \gamma & \delta
	\end{pmatrix}\in \mathcal{M}$. If $\alpha$ is invertible, then we define $\det_{K}(\theta) = -\gamma\alpha^{-1}\beta + \delta$ to be the $K-$determinant of $\theta$ and if $\delta$ is invertible, then  $\det_{H}(\theta) = \alpha - \beta\delta^{-1} \gamma$ is defined to be the $H-$determinant of $\theta$. In general, $\det_{H}(\theta)$ or $\det_{K}(\theta)$ need not be a group homomorphism.
	
	\begin{theorem}\label{t31}
		Let $G = H\rtimes K$ and $\theta = \begin{pmatrix}
		\alpha & \beta\\ \gamma & \delta
		\end{pmatrix}$ be an endomorphism of $G$. Then the following hold
		\begin{itemize}
			\item[$(i)$] if $\alpha$ is invertible and $\Delta_{K} = \det_{K}(\theta)$ is invertible, then $\theta$ is invertible and 
			\[\theta^{-1} = \begin{pmatrix}
			\alpha^{-1}-\alpha^{-1}\beta{\Delta_{K}}^{-1}(-\gamma\alpha^{-1}) & -\alpha^{-1}\beta{\Delta_{K}}^{-1}\\
			{\Delta_{K}}^{-1}(-\gamma\alpha^{-1}) & {\Delta_{K}}^{-1}
			\end{pmatrix}. \] 
			Moreover, $\det_{H}(\theta^{-1}) = \alpha^{-1}$ and $\det_{H}(\theta)$ is a group homomorphism.
			\item[$(ii)$] if $\delta$ is invertible and $\Delta_{H} = \det_{H}(\theta)$ is invertible, then $\theta$ is invertible and 
			\[\theta^{-1} = \begin{pmatrix}
			{\Delta_{H}}^{-1} & {\Delta_{H}}^{-1}(-\beta\delta^{-1})\\
			-\delta^{-1}\gamma{\Delta_{H}}^{-1}& \delta^{-1}\gamma{\Delta_{H}}^{-1}(-\beta\delta^{-1}) + \delta^{-1}
			\end{pmatrix}. \] 
			Moreover, $\det_{K}(\theta^{-1}) = \delta^{-1}$ and $\det_{H}(\theta)$ is a group homomorphism.
		\end{itemize}
	\end{theorem}
	\begin{proof}
		We will only prove the part $(i)$, as the proof of the part $(ii)$ will be on the same lines.
		
		Let $\alpha$ be invertible and $\Delta = \Delta_{K}$ invertible. Then $\Delta\Delta^{-1} = 1 = \Delta^{-1}\Delta$. This gives us
		\begin{equation}\label{e1}
		\gamma\alpha^{-1}\beta\Delta^{-1} + 1 = \delta\Delta^{-1}\; \text{and}\; \Delta^{-1}(-\gamma\alpha^{-1}\beta + \delta) = 1.
		\end{equation}
		Now, let $\alpha^{\prime} = \alpha^{-1}-\alpha^{-1}\beta{\Delta}^{-1}(-\gamma\alpha^{-1})$, $\beta^{\prime} = -\alpha^{-1}\beta{\Delta}^{-1}$, $\gamma^{\prime} = {\Delta}^{-1}(-\gamma\alpha^{-1})$ and $\delta^{\prime} = {\Delta}^{-1}$. Then we have
		\begin{align*}
		\alpha\alpha^{\prime} + (\beta\gamma^{\prime})^{\gamma\alpha^{\prime}} &= \alpha(\alpha^{-1}-\alpha^{-1}\beta{\Delta}^{-1}(-\gamma\alpha^{-1})) + (\beta{\Delta}^{-1}(-\gamma\alpha^{-1}))^{\gamma(\alpha^{-1}-\alpha^{-1}\beta{\Delta}^{-1}(-\gamma\alpha^{-1})
			)}\\
		&= \alpha(\alpha^{-1}-\alpha^{-1}\beta{\Delta}^{-1}(-\gamma\alpha^{-1})) + \alpha(\alpha^{-1}\beta{\Delta}^{-1}(-\gamma\alpha^{-1}))^{\gamma(\alpha^{-1}-\alpha^{-1}\beta{\Delta}^{-1}(-\gamma\alpha^{-1})
			)}\\
		&=  \alpha(\alpha^{-1}-\alpha^{-1}\beta{\Delta}^{-1}(-\gamma\alpha^{-1}) + \alpha^{-1}\beta{\Delta}^{-1}(-\gamma\alpha^{-1})) \\
		&= \alpha(\alpha^{-1})\\
		&= 1.\\
		\alpha\beta^{\prime} + (\beta\delta^{\prime})^{\gamma\beta^{\prime}} &= \alpha(-\alpha^{-1}\beta{\Delta}^{-1}) + (\beta{\Delta}^{-1})^{\gamma(-\alpha^{-1}\beta{\Delta}^{-1})} \\
		&= \alpha(-\alpha^{-1}\beta{\Delta}^{-1}) + \alpha(\alpha^{-1}\beta{\Delta}^{-1})^{\gamma(-\alpha^{-1}\beta{\Delta}^{-1})} \\
		&= \alpha(-\alpha^{-1}\beta{\Delta}^{-1} + \alpha^{-1}\beta{\Delta}^{-1})\\
		&= 0.\\
		\gamma\alpha^{\prime} + \delta\gamma^{\prime} &= \gamma(\alpha^{-1}-\alpha^{-1}\beta{\Delta}^{-1}(-\gamma\alpha^{-1})) + \delta({\Delta}^{-1}(-\gamma\alpha^{-1}))\\
		&= \gamma\alpha^{-1} + (-\gamma\alpha^{-1}\beta\Delta^{-1}+ \delta\Delta^{-1})(-\gamma\alpha^{-1})\\
		&= \gamma\alpha^{-1} + (-\gamma\alpha^{-1}\beta+ \delta)\Delta^{-1}(-\gamma\alpha^{-1})\\
		&= \gamma\alpha^{-1} + (-\gamma\alpha^{-1}), \;(\text{using the Equation (\ref{e1})}) \\
		&= 0.\\ 
		\gamma\beta^{\prime} + \delta\delta^{\prime} &= \gamma(-\alpha^{-1}\beta{\Delta}^{-1}) + \delta{\Delta}^{-1}\\
		&= (-\gamma\alpha^{-1}\beta + \delta){\Delta}^{-1}\\
		&= \Delta{\Delta}^{-1}, \;(\text{using the Equation (\ref{e1})})\\
		&= 1.
		\end{align*}
		Thus $\theta^{-1} = \begin{pmatrix}
		\alpha^{-1}-\alpha^{-1}\beta{\Delta_{K}}^{-1}(-\gamma\alpha^{-1}) & -\alpha^{-1}\beta{\Delta_{K}}^{-1}\\
		{\Delta_{K}}^{-1}(-\gamma\alpha^{-1}) & {\Delta_{K}}^{-1}
		\end{pmatrix}.$
		
		\vspace{.2cm}
		Since $\delta^{\prime} = \Delta^{-1}$ is invertible, $\det_{H}(\theta^{-1})$ is defined and $\det_{H}(\theta^{-1}) = \alpha^{\prime}- \beta^{\prime}{\delta^{\prime}}^{-1}\gamma^{\prime} = \alpha^{-1}-\alpha^{-1}\beta{\Delta}^{-1}(-\gamma\alpha^{-1}) - (-\alpha^{-1}\beta{\Delta}^{-1})\Delta({\Delta}^{-1}(-\gamma\alpha^{-1})) = \alpha^{-1}$. Moreover, $\theta^{-1}\in End(G)$, as the inverse of an endomorphism is also anendomorphism. This implies that ${\Delta_{K}}^{-1}\in End(K)$. Hence ${\Delta_{K}}$ is a group homomorphism
	\end{proof}
	\begin{theorem}\label{t32}
		Let $G = H\rtimes K$ and $\theta = \begin{pmatrix}
		\alpha & \beta\\ \gamma & \delta
		\end{pmatrix}\in End(G)$ such that $\alpha$ is invertible. Then $\theta$ is invertible if and only if $\det_{K}(\theta)$ is invertible.  
	\end{theorem}
	\begin{proof}
		Let $\alpha$ be invertible. First assume that $\theta$ is invertible. Now, let $k, k^{\prime}\in K$, such that $\det_{K}(\theta)(k) = \det_{K}(\theta)(k^{\prime})$. Then
		\begin{equation}\label{e7}
		(\gamma\alpha^{-1}\beta(k))^{-1}\delta(k) = (\gamma\alpha^{-1}\beta(k^{\prime}))^{-1}\delta(k^{\prime})
		\end{equation}  
		Now,
		\begin{align*}
		\begin{pmatrix}
		\alpha & \beta \\ \gamma & \delta
		\end{pmatrix}\begin{pmatrix}
		(\alpha^{-1}\beta(k))^{-1}\\ k
		\end{pmatrix} &= \begin{pmatrix}
		\alpha((\alpha^{-1}\beta(k))^{-1})\beta(k)^{\gamma((\alpha^{-1}\beta(k))^{-1})}\\ \gamma((\alpha^{-1}\beta(k))^{-1})\delta(k)
		\end{pmatrix}\\
		&= \begin{pmatrix}
		(\alpha(\alpha^{-1}\beta(k))^{-1})^{\gamma((\alpha^{-1}\beta(k))^{-1})}\beta(k)^{\gamma((\alpha^{-1}\beta(k))^{-1})}\\ \gamma((\alpha^{-1}\beta(k))^{-1})\delta(k)
		\end{pmatrix}\\
		&= \begin{pmatrix}
		(\beta(k)^{-1})^{\gamma((\alpha^{-1}\beta(k))^{-1})}\beta(k)^{\gamma((\alpha^{-1}\beta(k))^{-1})}\\ \gamma((\alpha^{-1}\beta(k))^{-1})\delta(k)
		\end{pmatrix}\\
		&= \begin{pmatrix}
		1\\ \gamma((\alpha^{-1}\beta(k))^{-1})\delta(k)
		\end{pmatrix}.
		\end{align*}
		Similarly, we get 
		\begin{equation*}
		\begin{pmatrix}
		\alpha & \beta \\ \gamma & \delta
		\end{pmatrix}\begin{pmatrix}
		(\alpha^{-1}\beta(k^{\prime}))^{-1}\\ k^{\prime}
		\end{pmatrix}  = \begin{pmatrix}
		1\\ \gamma((\alpha^{-1}\beta(k^{\prime}))^{-1})\delta(k^{\prime})
		\end{pmatrix}.
		\end{equation*}
		Using the Equation (\ref{e7}), we get
		\[\begin{pmatrix}
		\alpha & \beta \\ \gamma & \delta
		\end{pmatrix}\begin{pmatrix}
		(\alpha^{-1}\beta(k))^{-1}\\ k
		\end{pmatrix} = \begin{pmatrix}
		\alpha & \beta \\ \gamma & \delta
		\end{pmatrix}\begin{pmatrix}
		(\alpha^{-1}\beta(k^{\prime}))^{-1}\\ k^{\prime}
		\end{pmatrix}. \]
		Since $\theta$ is invertible, $k = k^{\prime}$. This shows that $\det_{K}(\theta)$ is injective.
		
		Now, let $k^{\prime}\in K$ be any element. Since $\theta$ is invertible, we have $hk\in G$ such that $\theta(hk) = k^{\prime}$. Then $\alpha(h)\beta(k)^{\gamma(h)}\gamma(h)\delta(k) = k^{\prime}$. Using the uniqueness of representation, we have $\alpha(h)\beta(k)^{\gamma(h)} = 1$ and $\gamma(h)\delta(k) = k^{\prime}$. This implies that $\alpha(h\alpha^{-1}\beta(k)) = 1$. Since $\alpha$ is invertible, $h\alpha^{-1}\beta(k) = 1$ and so, $h = (\alpha^{-1}\beta(k))^{-1}$. Using $\gamma(h)\delta(k) = k^{\prime}$, we get $(-\gamma\alpha^{-1}\beta + \delta)(k) = k^{\prime}$. Thus $\det_{K}(\theta)$ is surjective. Hence $\det_{K}(\theta)$ is invertible. The converse holds using the Theorem \ref{t31}.
	\end{proof}
	\begin{theorem}\label{t33}
		Let $G = H\rtimes K$ and $\theta = \begin{pmatrix}
		\alpha & \beta\\ \gamma & \delta
		\end{pmatrix}\in End(G)$ such that $\delta$ is invertible. Then $\theta$ is invertible if and only if $\det_{H}(\theta)$ is invertible.  
	\end{theorem}
	\begin{proof}
		The proof is on the lines of the proof of the Theorem \ref{t32}.
	\end{proof}

	Now, we define a subset of $End(G)$ contatining all the automorphisms of the group $G$, where $G$ is the semidirect product of the groups $H$ and $K$. Consider the set
	\begin{equation*}
	\mathcal{A} = \left\{\begin{pmatrix}
	\alpha & \beta\\ \gamma & \delta
	\end{pmatrix}
	\mid \begin{matrix}
	\alpha \in Map(H, H), & \beta\in Map(K, H)\\ \gamma\in Hom(H, K), & \text{and}\;\delta \in End(K)
	\end{matrix} \right\},
	\end{equation*}
	where the maps $\alpha$, $\beta$, $\gamma$ and $\delta$ satisfy the following
	\begin{itemize}
		\item[$(A_{1})$] $\alpha(hh^{\prime}) = \alpha(h)\alpha(h^{\prime})^{\gamma(h)}$,
		\item[$(A_2)$] $\beta(kk^{\prime}) = \beta(k)\beta(k^{\prime})^{\delta(k)}$,
		\item[$(A_3)$] $\gamma(h^{k})\delta(k) = \delta(k)\gamma(h)$,
		\item[$(A_{4})$] $\alpha(h^{k})\beta(k)^{\gamma(h^{k})} = \beta(k)\alpha(h)^{\delta(k)}$,
		\item[$(A_{5})$] for any $h^{\prime}k^{\prime}\in G$, there exists a unique $h\in H$ and $k\in K$ such that $h^{\prime} = \alpha(h)\beta(k)^{\gamma(h)}$ and $k^{\prime} = \gamma(h)\delta(k)$,
	\end{itemize}
	for all $h,h^{\prime}\in H$ and $k, k^{\prime}\in K$.
	
	Clearly, $\mathcal{A}$ is a submonoid of $\mathcal{M}$. Note that a one to one correspondence between $Aut(G)$ and $\mathcal{A}$ is already dicussed in \cite{bc} and \cite{hsu}. Now, we consider some important subsets of $\mathcal{A}$. Let
	\begin{align*}
	A &= \left\{\begin{pmatrix}
	\alpha & 0 \\ 0 & 1
	\end{pmatrix} \mid \alpha \in Aut(H), \alpha(h^{k}) = \alpha(h)^{k}\; \text{for all}\; h\in H \; \text{and}\; k\in K \right\},\\
	B &= \left\{\begin{pmatrix}
	1 & \beta \\ 0 & 1
	\end{pmatrix} \mid \beta \in CHom(K, Z(H)) \right\},\\
	C &= \left\{\begin{pmatrix}
	1 & 0 \\ \gamma & 1
	\end{pmatrix} \mid \gamma \in Hom(H, K), \gamma(h)\in C_{K}(H), \gamma(h^{k}) = \gamma(h)^{k}\; \text{for all}\; h\in H \; \text{and}\; k\in K \right\},\\
	D &= \left\{\begin{pmatrix}
	1 & 0 \\ 0 & \delta
	\end{pmatrix} \mid \delta \in Aut(K), k^{-1}\delta(k)\in C_{K}(H)\; \text{for all}\; k\in K \right\}.
	\end{align*}
	One can easily note that $A$, $B$ and $D$ are the subgroups of the group $\mathcal{A}$. However, $C$ need not be a subgroup of $\mathcal{A}$. It is easy to check that $A$ and $D$ normalizes $B$ and $C$ both.
	
	\begin{lemma}\label{le1}
		Let $\begin{pmatrix}
		1 & \beta\\ \gamma & 1
		\end{pmatrix}\mathcal{A}$. Then the following hold
		\begin{itemize}
			\item[$(i)$] $\begin{pmatrix}
			1-\beta\gamma & 0\\ 0 & 1
			\end{pmatrix}\in A$, 
			\item[$(ii)$] $\begin{pmatrix}
			1 & (1-\beta\gamma)^{-1}\beta \\ 0 & 1
			\end{pmatrix}\in B$.
%			\item[$(i)$] $1-\beta\gamma \in A$,
%			\item[$(ii)$]  $(1-\beta\gamma)^{-1}\beta \in B$.
		\end{itemize}
	\end{lemma}
	\begin{proof}
		\begin{itemize}
%			\item[$(i)$] Using the conditions $(A_{1})$ and $(A_{4})$, it is easy to see that $\gamma\in C$. 
%			\item[$(ii)$] Using $(A_{2})$, it is clear that $\beta\in CHom(K, H)$. Also, the conditions $(A_{1})$ and $(A_{4})$ implies that $\beta(K) \subseteq Z(H)$.
			\item[$(i)$]   	Let $\begin{pmatrix}
			1 & -\beta\\ 0 & 1
			\end{pmatrix}, \begin{pmatrix}
			1 & 0\\ \gamma & 1
			\end{pmatrix}\in \mathcal{A}$. Then 
			\begin{equation}\label{e2}
			\begin{pmatrix}
			1 & -\beta\\ 0 & 1
			\end{pmatrix} \begin{pmatrix}
			1 & 0\\ \gamma & 1
			\end{pmatrix} = \begin{pmatrix}
			1-\beta\gamma & -\beta \\ \gamma & 1
			\end{pmatrix}\in \mathcal{A}.
			\end{equation}
			Using $(A_{5})$, we get $1-\beta\gamma$ is a bijection. Now, we prove that $1-\beta\gamma$ is a group homomorphism. Let $h, h^{\prime}\in H$. Then
			\begin{align*}
			\beta\gamma(hh^{\prime}) &= \beta(\gamma(h)\gamma(h^{\prime}))\\
			&= \beta(\gamma(h))\beta(\gamma(h^{\prime}))^{\gamma(h)}\\
			&= \beta(\gamma(h))\beta(\gamma(h^{\prime})), \; (\text{using $(A_{3})$}).
			\end{align*}
			Thus $\beta\gamma$ is a group homomorphism. This implies that $1-\beta\gamma$ is a group homomorphism. Therefore, $1-\beta\gamma\in Aut(H)$. Now, for all $h\in H$ and $k\in K$, we have
			\begin{align*}
			\beta\gamma(h^{k}) &= \beta(\gamma(h^{k}))\\
			&= \beta(\gamma(h)^{k}), \; (\text{using $(A_{3})$})\\
			&= \beta(k\gamma(h)k^{-1})\\
			&= \beta(k\gamma(h))\beta(k^{-1})^{k\gamma(h)}\\
			&= \beta(k)\beta(\gamma(h))^{k}\beta(k^{-1})^{k\gamma(h)}\\
			&= \beta(\gamma(h))^{k}\beta(k){(\beta(k^{-1})^{\gamma(h)})}^{k}, \; (\text{because $\beta(K)\subseteq Z(H)$})\\
			&= \beta(\gamma(h))^{k}\beta(k)\beta(k^{-1})^{k}, \; (\text{using $(A_{1})$})\\
			&= \beta(\gamma(h))^{k}\beta(kk^{-1})\\
			&= \beta(\gamma(h))^{k}.
			\end{align*} 
			Therefore $(1-\beta\gamma(h^{k})) = (1-\beta\gamma(h))^{k}$ for all $h\in H$ and $k\in K$. Thus $\begin{pmatrix}
			1-\beta\gamma & 0\\ 0 & 1
			\end{pmatrix}\in A$.
			\item[$(ii)$] Using the part $(i)$, we get $\begin{pmatrix}
			(1-\beta\gamma)^{-1} & 0\\ 0 & 1
			\end{pmatrix}\in A$. Let $k, k^{\prime}\in K$. Then
			\begin{align*}
			(1-\beta\gamma)^{-1}\beta(kk^{\prime}) &= (1-\beta\gamma)^{-1}(\beta(k)\beta(k^{\prime})^{k})\\
			&= (1-\beta\gamma)^{-1}(\beta(k))(1-\beta\gamma)^{-1}(\beta(k^{\prime})^{k})\\
			&= ((1-\beta\gamma)^{-1}\beta)(k)((1-\beta\gamma)^{-1}\beta)(k^{\prime})^{k}.
			\end{align*}
			Therefore, $(1-\beta\gamma)^{-1}\beta\in CHom(K, H)$. Hence $\begin{pmatrix}
			1 & (1-\beta\gamma)^{-1}\beta \\ 0 & 1
			\end{pmatrix}\in B$.
%			Now for all $h\in H$ and $k\in K$, we have
%			\begin{align*}
%			h^{k}(1-\beta\gamma)^{-1}\beta(k) &= ((1-\beta\gamma)^{-1}(\beta(k)))^{h^{k}}h^{k}\\
%			&=  (1-\beta\gamma)^{-1}(\beta(k)^{h^{k}})h^{k}\\
%			&= (1-\beta\gamma)^{-1}(\beta(k))h^{k}, \; (\text{because $\beta(K)\subseteq Z(H)$}).
%			\end{align*}
%			Thus $(1-\beta\gamma)^{-1}\beta(K)\subseteq Z(H)$. Hence $\begin{pmatrix}
%			1 & (1-\beta\gamma)^{-1}\beta \\ 0 & 1
%			\end{pmatrix}\in B$.
		\end{itemize}
	\end{proof}
	\begin{theorem}
		Let $G= H\rtimes K$ and $\mathcal{A}$ be defined as above. Then $\mathcal{A} = ABCD$.
	\end{theorem}
	\begin{proof}
%		Let $\begin{pmatrix}
%		1 & -\beta\\ 0 & 1
%		\end{pmatrix}, \begin{pmatrix}
%		1 & 0\\ \gamma & 1
%		\end{pmatrix}\in \mathcal{A}$. Then 
%		\begin{equation}\label{e2}
%		\begin{pmatrix}
%		1 & -\beta\\ 0 & 1
%		\end{pmatrix} \begin{pmatrix}
%		1 & 0\\ \gamma & 1
%		\end{pmatrix} = \begin{pmatrix}
%		1-\beta\gamma & -\beta \\ \gamma & 1
%		\end{pmatrix}\in \mathcal{A}.
%		\end{equation}
%		This implies that $1-\beta\gamma$ is a bijection. Also, we have
Let $\begin{pmatrix}
1 & \beta\\ \gamma & 1
\end{pmatrix} \in \mathcal{A}$. Then using the Lemma \ref{le1}, we get
		\begin{equation}\label{e3}
		\begin{pmatrix}
		1 & \beta\\ \gamma & 1
		\end{pmatrix} = \begin{pmatrix}
		1-\beta\gamma & 0\\ 0 & 1
		\end{pmatrix}\begin{pmatrix}
		1 & (1-\beta\gamma)^{-1}\beta\\ 0 & 1
		\end{pmatrix}\begin{pmatrix}
		1 & 0\\ \gamma & 1
		\end{pmatrix}\in ABC.
		\end{equation}
		Now, let $\begin{pmatrix}
		\alpha & \beta \\ \gamma & \delta
		\end{pmatrix} \in \mathcal{A}$, where the maps satisfy the conditions $(A_{1}) - (A_{5})$. Then 
		\begin{equation}\label{e4}
		\begin{pmatrix}
		\alpha & \beta\\ \gamma & \delta
		\end{pmatrix} = \begin{pmatrix}
		\alpha & 0\\ 0 & 1
		\end{pmatrix}\begin{pmatrix}
		1 & \alpha^{-1}\beta\delta^{-1}\\ \gamma & 1
		\end{pmatrix}\begin{pmatrix}
		1 & 0\\ 0 & \delta
		\end{pmatrix}.
		\end{equation}
		Now, let $\begin{pmatrix}
		\alpha^{-1} & 0\\ 0 & 1
		\end{pmatrix}\in A$ and $\begin{pmatrix}
		1 & 0\\ 0 & \delta^{-1}
		\end{pmatrix}\in D$. Then for all $k,k^{\prime}\in K$, we have
		\begin{align*}
		\alpha^{-1}\beta\delta^{-1}(kk^{\prime}) &= \alpha^{-1}\beta(\delta^{-1}(k)\delta^{-1}(k^{\prime}))\\
		&= \alpha^{-1}(\beta(\delta^{-1}(k))\beta(\delta^{-1}(k^{\prime}))^{\delta(\delta^{-1}(k))})\\
		&= \alpha^{-1}(\beta(\delta^{-1}(k)))\alpha^{-1}(\beta(\delta^{-1}(k^{\prime}))^{k})\\
		&= \alpha^{-1}\beta\delta^{-1}(k)\alpha^{-1}\beta\delta^{-1}(k^{\prime})^{k}.
		\end{align*}
		Therefore, $\begin{pmatrix}
		1 & \alpha^{-1}\beta\delta^{-1}\\ 0 & 1
		\end{pmatrix}\in B$. By the Equations (\ref{e3}) and (\ref{e4}), we get
		\[\begin{pmatrix}
		\alpha & \beta\\ \gamma & \delta
		\end{pmatrix} = \begin{pmatrix}
		\alpha & 0\\ 0 & 1
		\end{pmatrix}\begin{pmatrix}
		1 & \alpha^{-1}\beta\delta^{-1}\\ \gamma & 1
		\end{pmatrix}\begin{pmatrix}
		1 & 0\\ 0 & \delta
		\end{pmatrix}\in A(ABC)D \subseteq ABCD. \]
		Thus $\mathcal{A} \subseteq ABCD$. Also, it is evident that $ABCD \subseteq \mathcal{A}$. Hence $\mathcal{A} = ABCD$.

	\end{proof}
	
	\begin{theorem}\label{t34}
		Let $G = H\rtimes K$ and $\theta = \begin{pmatrix}
		\alpha & \beta \\ \gamma & \delta
		\end{pmatrix}\in \mathcal{A}$. Then $\Delta_{H} = \det_{H}(\theta)$ is invertible if and only if $\Delta_{K}= \det_{K}(\theta)$ is invertible. Moreover, 
		\begin{align}
		{\Delta_{H}}^{-1} &= \alpha^{-1} - \alpha^{-1}\beta{\Delta_{K}}^{-1}(-\gamma\alpha^{-1}) \label{e5}\\
		\text{and}\hspace{.5cm} {\Delta_{K}}^{-1} &= \delta^{-1}\gamma{\Delta_{H}}^{-1}(-\beta\delta^{-1})+ \delta^{-1}. \label{e6}
		\end{align}
	\end{theorem}
	\begin{proof}
		Let $\Delta_{K}$ be invertible. Then we prove that $\Delta_{H}$ is invertible. For this, let $h\in H$ such that $\Delta_{H}(h) = 1$. This implies that $\alpha(h) = \beta{\delta}^{-1}\gamma(h)$. Now
		\begin{align*}
		\Delta_{K}(\delta^{-1}\gamma(h)) &= (-\gamma\alpha^{-1}\beta + \delta)(\delta^{-1}\gamma(h))\\
		&= (\gamma\alpha^{-1}(\beta\delta^{-1}\gamma(h)))^{-1}\gamma(h)\\
		&= (\gamma\alpha^{-1}(\alpha(h)))^{-1}\gamma(h)\\
		&= (\gamma(h))^{-1}\gamma(h)\\
		&= 1.
		\end{align*}
		Since $\Delta_{K}$ is invertible, $\delta^{-1}\gamma(h) = 1$. This implies that $\alpha(h) = \beta({\delta}^{-1}\gamma(h)) = \beta(1) = 1$. Since $\alpha$ is invertible, $h = 1$. Thus $\Delta_{H}$ is injective.
		
		Now, let $h\in H$ and $k\in K$ such that $\Delta_{K}(k) = (\gamma\alpha^{-1}(h))^{-1}$. Then $(\gamma\alpha^{-1}\beta(k))^{-1} = (\gamma\alpha^{-1}(h))^{-1}\delta(k)^{-1}$. Now
		\begin{align*}
		\Delta_{H}(\alpha^{-1}(h)({\alpha^{-1}\beta(k)})^{-1}) &= (\alpha - \beta\delta^{-1}\gamma)(\alpha^{-1}(h)({\alpha^{-1}\beta(k)})^{-1})\\ &= \alpha(\alpha^{-1}(h)({\alpha^{-1}\beta(k)})^{-1})({\beta\delta^{-1}(\gamma\alpha^{-1}(h)(\gamma{\alpha^{-1}\beta(k)})^{-1})})^{-1}\\
		&= h(\alpha(({\alpha^{-1}\beta(k)})^{-1}))^{\gamma\alpha^{-1}(h)}(\beta\delta^{-1}(\delta(k^{-1})))^{-1}\\
		&= h(\alpha(({\alpha^{-1}\beta(k)})^{-1}))^{\gamma\alpha^{-1}(h)}(\beta(k^{-1}))^{-1}\\
		&= h\gamma\alpha^{-1}(h)\alpha(({\alpha^{-1}\beta(k)})^{-1}){\gamma\alpha^{-1}(h)}^{-1}(\beta(k^{-1}))^{-1}\\
		&= h\gamma\alpha^{-1}(h)((\alpha({\alpha^{-1}\beta(k)}))^{-1})^{\gamma{\alpha^{-1}\beta(k)}^{-1}}{\gamma\alpha^{-1}(h)}^{-1}(\beta(k^{-1}))^{-1}\\
		&= h\gamma\alpha^{-1}(h)\gamma{\alpha^{-1}\beta(k)}^{-1}(\beta(k))^{-1}\gamma{\alpha^{-1}\beta(k)}{\gamma\alpha^{-1}(h)}^{-1}(\beta(k^{-1}))^{-1}\\
		&= h\delta(k)^{-1}\beta(k)^{-1}\delta(k)(\beta(k^{-1}))^{-1}\\
		&= h(\beta(k^{-1})\beta(k)^{\delta(k^{-1})})^{-1}\\
		&= h(\beta(k^{-1}k))^{-1}\\
		&= h.
		\end{align*}
		Thus for $h\in H$, there exits $\alpha^{-1}(h)({\alpha^{-1}\beta(k)})^{-1}\in H$ such that $\Delta_{H}(\alpha^{-1}(h)({\alpha^{-1}\beta(k)})^{-1}) = h$.
		Therefore, $\Delta_{H}$ is a bijection and so, $\Delta_{H}$ is invertible.  
		
		Further, $\Delta_{H}(\alpha^{-1}(h)({\alpha^{-1}\beta(k)})^{-1}) = h$ gives that ${\Delta_{H}}^{-1}(h) = (\alpha^{-1} - \alpha^{-1}\beta{\Delta_{K}}^{-1}(-\gamma\alpha^{-1}))(h)$. Hence ${\Delta_{H}}^{-1} = \alpha^{-1}-\alpha^{-1}\beta{\Delta_{K}}^{-1}(-\gamma\alpha^{-1})$. Using the similar arguments, one can easily prove the converse part and the Equation (\ref{e6}).
	\end{proof}
%	\begin{remark}
%		Let $G = H\rtimes K$. Then we define the determinant of an element $\theta = \begin{pmatrix}
%		\alpha & \beta\\ \gamma & \delta
%		\end{pmatrix} \in \mathcal{A}$ as $\det(\theta) = \alpha-\beta\delta^{-1}\gamma$ and denote it by $\Delta$.
%	\end{remark}

	\begin{theorem}
		Let $G = H\rtimes K$ and $\theta= \begin{pmatrix}
		\alpha & \beta\\ \gamma & \delta
		\end{pmatrix}\in \mathcal{A}$. If $\Delta = \det(\theta)$ is invertible, then 
		\[\theta^{-1} = \begin{pmatrix}
		\Delta^{-1} & \Delta^{-1}(-\beta\delta^{-1})\\
		-\delta^{-1}\gamma\Delta^{-1}& \delta^{-1}\gamma\Delta^{-1}(-\beta\delta^{-1}) + \delta^{-1}
		\end{pmatrix}.\]
		Moreover, $\theta^{-1}\in \mathcal{A}$ and $\det(\theta^{-1}) = \alpha^{-1}$.
	\end{theorem}
	\begin{proof}
		Let $\Delta$ be invertible. Then the formula for $\theta^{-1}$ can be obtained directly from Theorem $\ref{t31}$. Using Theorem \ref{t34}, we have $\det_{K}(\theta)$ is invertible. Also, using the two formulas for $\theta^{-1}$ in Theorem \ref{t31}, we get ${\det_{K}}^{-1}(\theta) = \delta^{-1}\gamma\Delta^{-1}(-\beta\delta^{-1})+ \delta^{-1}$. Thus $\theta^{-1}\in \mathcal{A}$. Further, using $\theta = (\theta^{-1})^{-1}$, we get $\det(\theta^{-1}) = \alpha^{-1}$.
	\end{proof}
	
	By Theorem \ref{t34}, we have ${\Delta_{K}}^{-1} = \delta^{-1}\gamma\Delta^{-1}(-\beta\delta^{-1})+ \delta^{-1}$. This gives an eligant form for the inverse automorphisms in $\mathcal{A}$.
	
	\begin{theorem}
		Let $G = H\rtimes K$ and $\theta= \begin{pmatrix}
		\alpha & \beta\\ \gamma & \delta
		\end{pmatrix}\in \mathcal{A}$. If $\det_{H}(\theta)$(or $\det_{K}(\theta)$) is invertible, then 
		\[\theta^{-1} = \begin{pmatrix}
		{\Delta_{H}}^{-1} & -\alpha^{-1}\beta{\Delta_{K}}^{-1}\\
		-\delta^{-1}\gamma{\Delta_{H}}^{-1}& {\Delta_{K}}^{-1}
		\end{pmatrix}.\]
	\end{theorem}

\end{document}